\documentclass[12pt]{amsart}
\usepackage[utf8]{inputenc}
\usepackage{graphicx,hyperref}
\usepackage{subfiles}
\usepackage{amsmath}
\usepackage{amssymb}
\usepackage[utf8]{inputenc}
\usepackage[enableskew]{youngtab}
\usepackage{ytableau}
\usepackage[margin=1in]{geometry}
\usepackage{xcolor}
\usepackage{soul}
\usepackage{tikz}
\usepackage{latexsym,amscd,amssymb,epsfig}
\usepackage{epstopdf}
\usepackage{varwidth}
\usepackage{multicol}
\usepackage{young}
\usepackage{caption}
\usetikzlibrary{calc}

\newcommand{\SHS}{\mathrm{SHS}}

\newcommand{\sh}{{\rm sh\,}}

\newcommand{\ov}{\overline}

\newcommand{\cho}{\stackrel {\rm  Ch}{\leq}} 
\newcommand{\chos}{\stackrel {\rm  Ch}{<}}

\newcommand{\cellsize}{15}
\newlength{\cellsz} 
\newsavebox{\cell}
\sbox{\cell}{\begin{picture}(\cellsize,\cellsize)
\put(0,0){\line(1,0){\cellsize}} \put(0,0){\line(0,1){\cellsize}}
\put(\cellsize,0){\line(0,1){\cellsize}}
\put(0,\cellsize){\line(1,0){\cellsize}}
\end{picture}}
\newcommand\cellify[1]{\def\thearg{#1}\def\nothing{}%
\ifx\thearg\nothing \vrule width0pt height\cellsz depth0pt\else
\hbox to 0pt{\usebox{\cell} \hss}\fi%
\vbox to \cellsz{ \vss \hbox to \cellsz{\hss$#1$\hss} \vss}}
\newcommand\tableau[1]{\vtop{\let\\\cr
\baselineskip -16000pt \lineskiplimit 16000pt \lineskip 0pt
\ialign{&\cellify{##}\cr#1\crcr}}}

\theoremstyle{plain}
   \newtheorem{theorem}{Theorem}[section]
   \newtheorem{proposition}[theorem]{Proposition}

   \newtheorem{lemma}[theorem]{Lemma}
\theoremstyle{definition}
   \newtheorem{definition}[theorem]{Definition}
   \newtheorem{example}[theorem]{Example}

\numberwithin{equation}{section}

\newcommand\sym{{\mathfrak{S}}}
\newcommand\len{{\mathrm{length}}}

\newcommand\Knuth{{\underset{K}{\sim}}}
\newcommand\dKnuth{{\underset{d-K}{\sim}}}

\newcommand\Inv{{\mathrm{Inv}}}
\newcommand\Des{{\mathrm{Des}}}
\newcommand\Par{{\mathrm{Par}}}

\definecolor{scrippsgreen}{rgb}{.24,.55,.24}

\begin{document}

\title{A New Partial Order on SYT}
\author{Gizem Karaali}
\address{Department of Mathematics, Pomona College, Claremont, CA 91711, USA}
\email{gizem.karaali@pomona.edu}
\author{Isabella Senturia}
\address{Department of Mathematics, Pomona College, Claremont, CA 91711, USA}
\email{iesa2016@mymail.pomona.edu}
\author{M\"{u}ge Ta\c{s}kin}
\address{Department of Mathematics, Bo\v{g}azi\c{c}i \'{U}niversitesi, 34342 Bebek, Istanbul, Turkey}
\email{muge.taskin@boun.edu.tr}

\date{\today}

\begin{abstract}
We define a new partial order on $SYT_n$, the set of all standard Young tableaux with $n$ cells, by combining the chain order with the notion of horizontal strips. We prove various desirable properties of this new order.
\end{abstract}



\maketitle

\section{Introduction}

There are several partial orders defined on the complete set $SYT_n$ of standard Young tableaux of $n$ boxes, including the chain order (established via  the dominance order on  restricted tableaux), the geometric order (related to the preorder on $\sym_n$ induced from the geometric order on the orbital varieties associated to the Lie algebra $\mathfrak{sl}_n$), and the Kazhdan-Lusztig (KL) order (a combinatorial equivalence relation on permutations stemming from KL polynomials \cite{Kazhdan-Lusztig} related to
the theory of primitive ideals (cf.\ \cite{Joseph, Vogan})).
Melnikov
introduced additional orders, such as the induced Duflo order (also known as the weak order) \cite{Melnikov1}, the Duflo-Chain order, and the Vogan-Chain order \cite{Melnikov4}, in order to provide combinatorial definitions of the KL and geometric orders. Taskin \cite{Taskin} interpreted several combinatorial characterizations of the geometric order, chain order, KL order, and  weak order.

In this paper, we introduce a new partial order on $SYT_n$, which we call {\it the chain-strip order}. Our ultimate goal is to converge to  combinatorially  intuitive descriptions of the various partial orders already defined on $SYT_n$; along the way, we explore the intricacies of other natural combinatorial constructions on Young tableaux.

This short paper is organized as follows: In Section \ref{S:ChainStripOrder} we introduce the chain-strip order: it is defined by juxtaposing an intuitive shape condition with the deconstruction of a given standard Young tableau into horizontal strips of boxes.
In Section \ref{properties of the order}, we
explore the various desirable properties of the chain-strip order. In particular, we show that it is preserved by a handful of standard combinatorial operations and is compatible with several natural constructions. Section \ref{remaining questions} concludes the paper. Here we explore two more properties that might be desirable to have for a partial order on $SYT_n$; this time these are not quite satisfied by the chain-strip order but our observations allow us to propose possible promising avenues of future work.

\section{The Chain-Strip Order on $SYT_n$} 
\label{S:ChainStripOrder}

After a quick overview of basic notation (\S\S\ref{SS:PartitionsETC}), we define the notion of a sequence of horizontal strips (\S\S\ref{SS:SHSDefinition}) and use it to define the chain-strip order (\S\S\ref{SS:DefineC-S}). The final subsection (\S\S\ref{descreint}) offers a recharacterization of the chain-strip order in terms of descent sets.

\subsection{Partitions, permutations and tableaux} 
\label{SS:PartitionsETC}

 A \emph{partition} $\lambda$ of a natural number \textit{n}, also known as an \emph{integer partition}, is an ordered tuple of positive numbers $(\lambda_1, \lambda_2, \cdots, \lambda_k)$ such that $\lambda_1 + \lambda_2 + \cdots +\lambda_k = n$ and $\lambda_1 \geq \lambda_2 \geq \cdots \geq \lambda_k$. 
 There is a well-known order on partitions called the \emph{dominance order}. For our purposes, we will define here the \emph{opposite dominance order}:

\begin{definition}
\label{D:OppositeDominance}
Let $\lambda=(\lambda_1, \lambda_2, ..., \lambda_k), \mu=(\mu_1, \mu_2, ..., \mu_l)$ be two  partitions of $n$.
We say that $\lambda \leq_{dom}^{opp} \mu$ in the opposite dominance order, if 
		$$
	 \lambda_1 + \lambda_2 + ... + \lambda_i \geq \mu_1 + \mu_2 +... + \mu_i  \text{ for all } 1 \leq i \leq min(k,l).$$
\end{definition}
 
 \textit{Young diagrams} are all possible arrangements of \textit{n} boxes into rows and columns, with the number of boxes in each subsequent row weakly decreasing. We associate a partition to a given Young diagram as a descriptor for the number of boxes per row. 
 
For a partition $\lambda$ of $n$, a \textit{standard Young tableau $S$ of shape $\lambda$} is built from the Young diagram of shape $\lambda$ by filling it with the numbers 1 to $n$, each occurring exactly once and strictly increasing across rows (left to right) and down columns. In this case, we call $\lambda$ the \emph{shape of $S$} and denote it by $\sh(S)$. We denote the set of all standard Young tableaux made by filling in Young diagrams with $n$ boxes $SYT_n$. In what follows we will also need the notion of a {\it partial tableau}, that is, a Young diagram filled with distinct integers increasing across rows and down columns. Then a partial tableau $T$ of $n$ boxes is a standard Young tableau if and only if its entries are precisely the integers $1, \cdots, n$.

Given $\pi \in S_n$, the \textit{Robinson-Schensted-Knuth} algorithm bijectively assigns to it a pair of tableaux $(P(\pi), Q(\pi)) \in SYT_{n}\times SYT_{n}$; $P(\pi)$ is called the \emph{insertion tableau} of $\pi$ and $Q(\pi)$ is called the \emph{recording tableau} of $\pi$. The insertion tableau $P(\pi)$ can be obtained inductively via a sequence of {\it Robinson-Schensted row insertions}; we will denote the row insertion into a partial tableau $T$ of the integer $a$ not among the entries of $T$, by $r_a(T)$. One can analogously define Robinson-Schensted column insertions, denoted by $c_a(T)$; see \cite{Sagan} for details.

For each tableau $S \in SYT_n$, there is a unique associated set consisting of permutations whose insertion tableau is $S$.  The collection of these sets induces an equivalence relation on the set of permutations in $S_n$. Knuth \cite{Knuth} established another equivalence relation $\underset{K}{\sim}$ on $S_n$ where two permutations $\pi,\sigma \in S_n$ have  the same insertion tableau if and only if $\pi \underset{K}{\sim} \sigma$.\footnote{We say  $\sigma, \tau \in S_{n}$ differ by a single  {\it Knuth relation},  if
\[\sigma=x_{1} \ldots x_{i-1}{x}_{i}{x}_{i+1} x_{i+2} \ldots x_{n} \textmd{ and }
	\tau=x_{1} \ldots x_{i-1}x_{i+1}{x}_{i}x_{i+2}\ldots x_{n},\]
	where either $x_{i-1}$ or $x_{i+2}$  lies between $x_{i}$ and $x_{i+1}$. Two permutations are called  {\it Knuth
		equivalent}, written $\sigma  ~\Knuth~ \tau$, if one of them can be obtained
	from the other by applying a sequence of  Knuth relations. Analogously, one can define dual Knuth relations: see Footnote \ref{FN:DualKnuth}.}
So, the set of all permutations having the same insertion tableaux $S$  is called the \textit{Knuth class of $S$} and denoted by $\{\kappa_S\}_{S\in SYT_n}$.
For more details see \cite{Sagan}.

The Knuth equivalence relation and Knuth classes are relevant for the definition of the weak order. To prepare for the comparison that we make between the order we will define in this paper and the weak order, we recall the definition of the weak Bruhat order here.

\begin{definition}
The ({\it right}) {\it weak Bruhat order}, $\leq_{weak}$, on  $S_n$ is obtained by taking the transitive closure of the following relation: $$\sigma \leq_{weak} \tau  ~\text{if}~ \tau=\sigma \cdot s_i  \text{ and } \len(\tau)=\len(\sigma)+1$$ where $s_i$ denotes  the adjacent transposition $(i, i+1)$ and $\len(\sigma)$ measures the size of a reduced word of $\sigma$. 
\end{definition}

\noindent
{Alternatively, the weak Bruhat order on permutations can be characterized \cite[Prop. 3.1]{Bjorner} in terms of (left) inversion sets as follows: $$\sigma \leq_{weak} \tau ~\text{if and only if} ~~ \Inv_L(\sigma)\subset
\Inv_L(\tau)$$ where $\Inv_L(\sigma):=\{(i,j): 1 \leq i < j \leq n ~\text{
	and }~
	\sigma^{-1}(i) > \sigma^{-1}(j)\}$.
	}

The weak Bruhat order induces an order on $SYT_n$, called the weak order (first introduced by Melnikov \cite{Melnikov1} under the name {\it induced Duflo order}):
	
	\begin{definition}
		\label{weak-order-def}
		The {\it weak order} $(SYT_n,\leq_{weak})$ is the partial order induced by taking the transitive closure of the following relation:
		$$ \begin{aligned}
		S\leq_{weak} T & \mbox{ if there exist } ~\sigma\in \kappa_S, ~\tau\in \kappa_T \mbox{ such that } ~\sigma\leq_{weak} \tau.
		\end{aligned}
		$$
	\end{definition}

\subsection{Sequences of horizontal strips (SHS)}
\label{SS:SHSDefinition}

Inspired by the growth diagrams of Fomin \cite{Fomin}, and following Roby {\it et al}. \cite{Roby}, we first make the following definitions:

Identifying partitions $\lambda = (\lambda_1, \lambda_2, \cdots, \lambda_l)$ with shapes (Young diagrams) as we have done in the previous section, we can induce a partial order on shapes determined by ``componentwise comparison of sequences'', or equivalently, by reverse lexicographic order: {Given $\alpha = (\alpha_1, \alpha_2, \cdots, \alpha_m)$ and $\beta = (\beta_1, \beta_2, \cdots, \beta_l)$, we will say that $\alpha \subset \beta$ if and only if $ m \le l$ and  $\alpha_i \le \beta_i$ for all $i \le m$.}
Then for $\alpha \subset \beta$, we can define the {\it skew shape} $\beta \slash\alpha$ consisting of boxes in $\beta$ that are not in $\alpha$. 
A {\it skew shape} $\beta \slash \alpha$ is called a {\it horizontal strip} if no column of $\beta \slash \alpha$ contains more than one box. 

A tableau $T \in SYT_n$ with shape $\beta$ then gives rise to a {\it chain} of tableaux $0 = T^0 \subset T^1 \subset \cdots \subset T^k = T$ with a corresponding chain $\emptyset = \beta^0 \subset \beta^1 \subset \cdots \subset \beta^k = \beta$ of shapes, where the successive skew shapes $\beta^i \slash \beta^{i-1}$, $1 \le i \le k$, are horizontal strips{; we will fill these skew shapes with consecutive integers in such a way that all together they constitute the tableau $T$}. In this way we can ``grow'' the tableau $T$ by starting with the empty tableau and adding horizontal strips filled with the right numbers at each step. In the following we will be interested in the coarsest possible chain of shapes $\emptyset = \beta^0 \subset \beta^1 \subset \cdots \subset \beta^k = \beta$ (or equivalently, of tableaux) for $T$ that still satisfies the condition that $\beta^i\slash \beta^{i-1}$ all be horizontal strips: this corresponds to minimizing the {\it length} $k$ of the chain. We will call that minimal-length chain {\it the growth chain} of $T$.

\begin{example}
\label{E:FirstSHSTableau}
Let: 
\[ T=\tableau{1& 2 & 4&6 \\  3& 5\\7\\8 }.\]
Then we can ``grow'' the tableau $T$ by adding horizontal strips one after another, as follows:
\[ {\emptyset \;\; \subset \;\; \tableau{{\textbf{1}}& {\textbf{2}}} \;\; \subset \;\; \tableau{1& 2&{\textbf{4}}\\ {\textbf{3}}}\;\; \subset \;\; \tableau{1& 2 & 4&{\textbf{6}} \\  3&{\textbf{5}}} \;\; \subset \;\; \tableau{1& 2 & 4&6 \\  3& 5\\{\textbf{7}}}\;\; \subset \;\; \tableau{1& 2 & 4&6 \\  3& 5\\7\\{\textbf{8}}} = T.} \]

\newpage
Here we have the growth chain $\beta^0 \subset \beta^1 \subset \beta^2 \subset \beta^3 \subset \beta^4 \subset \beta^5 = \beta$, where the successive skew shapes $\beta^i \slash \beta^{i-1}$, $1 \le i \le 5$, are all horizontal strips. 
%\[ \small{\tableau{{\color{blue}{1}}& {\color{blue}{2}}} \quad , \quad \tableau{\,&\,&{\color{blue}{4}}\\ {\color{blue}{3}}}\quad , \quad \tableau{\,&\,&\,&{\color{blue}{6}} \\\,&{\color{blue}{5}}} \quad , \quad \tableau{\,&\,&\,&\,\\ \,&\,\\{\color{blue}{7}}}\quad , \quad \tableau{\,&\,&\,&\,\\ \,&\,\\\,\\{\color{blue}{8}}}.} \]
Note that we can identify this collection of horizontal strips with a sequence of sequences of consecutive integers:
\[ ((1,2),(3,4),(5,6),(7),(8))\]
In this manner the growth chain for our tableau $T$ corresponds to a sequence of sequences of consecutive integers describing the particular filling of the Young diagram of shape $\beta$ that yields $T$.
\end{example}

The example above motivates the following definition:

\begin{definition} \label{P-shsdef}
Let $S \in SYT_n$ and let $\emptyset = \beta^0 \subset \beta^1 \subset \cdots \subset \beta^k = \sh(S)$ be the growth chain of $S$. The \emph{sequence of horizontal strips} (SHS) of \textit{S} is the sequence $s = (s_1, s_2, \cdots, s_k)$ consisting of $k$ sequences of consecutive numbers given by:
\[s_i = (|\beta^{i-1}|+1, \cdots, |\beta^i|),\]
where $|\beta|$ is the number of boxes of the Young diagram of shape $\beta$.
\end{definition}

\begin{example} Note that two distinct tableaux with the same shape might have the same SHS. 
For example the two tableaux:
\[ S_1= \tableau{1&2&4\\3&6\\5} \qquad \textmd{ and } \qquad S_2= \tableau{1&2&6\\3&4\\5} 
\]
of shape $(3,2,1)$ have the same SHS: $((1,2),(3,4),(5,6))$.
\end{example}

In the following we will often want to compare two SHS. To do this, we first associate a set partition to a given SHS as follows: If $s = (s_1, s_2, \cdots, s_k)$ is the SHS of $S \in SYT_n$, consisting of $k$ sequences, with $s_i$ being made up of $|s_i|$ consecutive numbers, then we define 
the {\it set partition} $\mathrm{SP}(s)$ of $s$ to be
\small 
\[\mathrm{SP}(s) = \{\{1, 2, \cdots, |s_1|\},\{|s_1|+1, |s_1|+2,\cdots, |s_1|+|s_2|\},\cdots,\{|s_1|+|s_2|+\cdots+|s_{k-1}|+1, \cdots n\}\}. \] \normalsize
That is, $\mathrm{SP}$ is the forgetful functor that assigns to each sequence of horizontal strips the underlying set partition of $[n] = \{1,2,\cdots,n\}$.

Recall that a set partition $\alpha$ of the set $[n]$ is a {\it refinement} of a set partition $\beta$ of $[n]$ (and we write $\alpha \le \beta$) if every element of $\alpha$ is a subset of some element of $\beta$ (and we say that $\alpha$ is {\it finer} than $\beta$ and $\beta$ is {\it coarser} than $\alpha$). 
We can now carry this notion over to the world of SHS. In particular if $s$ and $s^{\prime}$ are the respective SHS for two tableaux $S, S^{\prime} \in SYT_n$, we will say that $s$ is a {\it refinement} of $s^{\prime}$ (and we will write $s \le s^{\prime}$) if and only if $\mathrm{SP}(s)$ is a refinement of $\mathrm{SP}(s^{\prime})$. In the following we will need the notion of a {\it one-step refinement}, which will correspond to the case where $s \le s^{\prime}$ and $|s|= |s^{\prime}|+1$.

So far we have been working with tableaux and we have defined a new notion, the SHS, for any given tableau. We can define an analogous notion for permutations $w\in S_n$.\linebreak  To do this we first recall that given a permutation $\pi=x_1x_2...x_n \in S_n$ written in one-line notation, an {\it increasing subsequence of length $k$} of $\pi$ is a sequence $(x_{i_1},x_{i_2}, \cdots,x_{i_k})$ such that 
\[ x_{i_1} < x_{i_2} < \cdots < x_{i_k}, \]
where $i_1 < i_2 < \cdots < i_k$.
We call an increasing subsequence  {\it an interval } if  it consists of consecutive integers. 
Now we can make the following definition:

%Set of disjoint, consecutively increasing subsequences of \pi is called SHS of \pi if its length is minimal.

\begin{definition}\label{P-D:SHS-word}
Let  $\pi=x_1x_2...x_n \in S_n$, written in one-line notation. Then the sequence $(\pi^1,\pi^2,...,\pi^k)$ is called the \emph{sequence of intervals} $\mathrm{SI}(\pi)$ of $\pi$ if $\pi^1,\pi^2,...,\pi^k$ is the complete list of maximal intervals in $\pi$ ordered in a way that removing parentheses in $(\pi^1,\pi^2,...,\pi^k)$ results in the ordered set $1,2,\ldots,n$.

\end{definition}

\begin{example}
    For $\pi = (1263574)$, the sequence of intervals $\mathrm{SI}(\pi)$ of $\pi$ is $((1,2,3,4), (5),(6,7))$. 
\end{example}

Note that if $\pi~\Knuth~\pi'$ for $\pi, \pi' \in S_n$, then $\mathrm{SI}(\pi)=\mathrm{SI}(\pi')$: one Knuth relation cannot change the sequence of intervals, as no integers exist between $i$ and $i+1$, $i$ and $i+1$ can never be swapped in a Knuth relation.

Definitions \ref{P-shsdef} and \ref{P-D:SHS-word} are  interrelated through the RSK algorithm. To see this recall that for a tableau $S \in SYT_n$,  with rows $R_1, R_2,...,R_k$, where $R_i$ represents the contents of the  $i^{th}$ row from the top read from left to right, the \textit{row word of $S$} is the permutation $$\pi_S=R_kR_{k-1}...R_1.$$

\begin{example}\label{P-shsexamp}
    Now consider the row word  $\pi_S = (536912478)$ of the tableau \[ S= \tableau{1&2&4&7&8\\3&6&9\\5}, \]  Then the sequence of intervals $\mathrm{SI}(\pi_S)$ of $\pi_S$ is $((1,2),(3,4),(5,6,7,8),(9))$, which corresponds directly to $\SHS(S)$, the tableau's SHS.
\end{example}

More generally, if $S \in SYT_n$, then it is straightforward to see that the sequence of intervals $SI(\pi_S)$ for the row word $\pi_S \in S_n$ of $S$ is the same as $\SHS(S)$. Then together with the fact that Knuth equivalence preserves sequences of intervals, this implies that the sequence of intervals of any permutation in a tableau's Knuth class is the same as the SHS of that tableau.
We summarize this in the following:
\begin{proposition}\label{P:SIsameasSHS}
Let $\pi$ be a permutation in the Knuth class of some standard Young tableau $S \in SYT_n$. Then the sequence of intervals $\mathrm{SI}(\pi)$ of $\pi$ and the sequence of horizontal strips $\SHS(S)$ of $S$ are the same.
\end{proposition}

In light of this, we will use the terms ``sequence of horizontal strips'' of a permutation $\pi$, $\SHS(\pi)$, and ``sequence of intervals'' of $\pi$, $\mathrm{SI}(\pi)$, interchangeably.

\subsection{A shape condition and a partial order}
\label{SS:DefineC-S}

For any $S\in SYT_n$ and for any $1\leq i<j\leq n$, we denote by $S_{[i,j]}$ the tableau obtained by first restricting $S$ to the integers in the interval  $[i,j]$ and  then  standardizing the resulting skew tableau  by jeu de taquin backward slides (see \cite{Sagan} for details).

In \cite{Melnikov4}, Melnikov defines the chain order on $SYT_n$ as follows:

\begin{definition}
For $S,T\in SYT_n$ put $S\cho T$
if
\begin{itemize}
\item[(i)] for any $i,j$ such that $1\leq i<j\leq n$, {one has $\sh(S_{[i,j]})\leqq_{dom}^{opp} \sh(T_{[i,j]})$}
\item[(ii)] if for some $i,j$ with $1\leq i<j\leq n$ one has $\sh(S_{[i,j]})= \sh(T_{[i,j]})$, then for any $k,l:\ i\leq k<l\leq j$
one has $\sh(S_{[k,l]})= \sh(T_{[k,l]})$.
\end{itemize}
\end{definition} 

{The following recharacterization will be useful to us in the rest of this paper:}

\begin{proposition}
\label{P:ChainOneLiner}
For $S, T \in SYT_n$, $S\cho T$ if and only if
for all $1\leq i < j \leq n$, either  $\sh(S_{[i,j]}) \lvertneqq_{dom}^{opp} \sh(T_{[i,j]}) $ or  $S_{[i,j]}= T_{[i,j]}$. 
\end{proposition}

Combining the above with the ideas of Section \ref{SS:SHSDefinition}, we define a new order. 

\begin{definition}\label{NewSHSOrderDef}
The {\it chain-strip order} on $SYT_n$ is defined by taking the transitive closure  of the following relation:   $S \le_{C-S} T$  if
\begin{enumerate}
\item For all $1\leq i < j \leq n$, either  $\sh(S_{[i,j]}) \lvertneqq_{dom}^{opp} \sh(T_{[i,j]}) $ or  $S_{[i,j]}= T_{[i,j]}$, and
\item  $\SHS(T)$ is equal to, or a one-step 
refinement of,  $\SHS(S)$. 
\end{enumerate}

\end{definition}

\begin{example}
    Let $S, T \in SYT_6$ be given as $$ S=\tableau{1&3&4\\2&5&6}\qquad \textmd{ and } \qquad T= \tableau{1&3&4\\2&6\\5}.$$ It is easy to see that $S$ and $T$ satisfy the first condition of Definition \ref{NewSHSOrderDef}.  Since  $\SHS(S)=((1),(2,3,4),(5,6)) =\SHS(T)$, we have $S \leq_{C-S} T.$
\end{example}

Partial orders on $SYT_n$ are interesting for various reasons but one of our main motivations in this project is to approximate as much as possible the weak order described in Section \ref{SS:PartitionsETC}. To this end, we state and prove the following:

\begin{proposition} \label{weakimpliesSHS}
For $S,T\in SYT_n$, if $S <_{weak}  T$, then $S \le_{C-S} T$. That is, the chain-strip order is stronger than the weak order. 
\end{proposition}

\begin{proof}
We know already (see e.g. \cite{Melnikov4} or \cite{Taskin}) that $S <_{weak}  T$ implies $S \cho T$. This takes care of the first condition in Definition \ref{NewSHSOrderDef}. To get the second condition,
it suffices to check the case where there exist	$\sigma\in \kappa_S, \tau\in \kappa_T \text{ such that } \sigma\leq_{weak} \tau		$.
Then $\sigma$ and $\tau$ differ by a single swap of two adjacent numbers $x_i, x_{i+1}$ for some $1\leq i <n$. That is, $\sigma=x_{1} \ldots x_{i-1}{x}_{i}{x}_{i+1} x_{i+2} \ldots x_{n}$ and $\tau=x_{1} \ldots x_{i-1}x_{i+1}{x}_{i}x_{i+2}\ldots x_{n}$. If $x_i$ and $x_{i+1}$ are not consecutive (i.e., $|x_i-x_{i+1}| > 1$),  then $\mathrm{SI}(\sigma) = \mathrm{SI}(\tau)$ and so $\SHS(S) = \SHS(T)$.  If $x_i$ and $x_{i+1}$ are consecutive, then we can assume due to the chain condition that $x_i = j$ and $x_{i+1} = j+1$ for some $1\leq j<n$. 
Then the single swap from $\sigma$ to $\tau$ splits the  interval containing $j$ and $j+1$ in $\mathrm{SI}(\sigma)$ to become two separate intervals in $\mathrm{SI}(\tau)$.
In other words $\SHS(T)$ is a one-step refinement of $\SHS(S)$.
\end{proof}

Computations in {\tt Mathematica} have yielded that the chain-strip order agrees with the weak order till $n=7$, at which point there are counterexamples. That is, we can find $S,T \in SYT_7$ such that $S\le_{C-S}T$ but $S \nleq_{weak} T$; see Section \ref{comparing} for more details. 
Nonetheless, it turns out that the chain-strip order has many of the desirable properties of the weak order, which we explore in Section \ref{properties of the order}.

\subsection{A reinterpretation in terms of descent sets} \label{descreint}

Recall that for $\pi \in S_n$ and $T \in SYT_n$, a handful of inversion and descent sets are defined as follows
\[ \Des_L(\pi):= \{ i\, ; \, 1 \le i \le n-1 \textmd{ and } \pi^{-1}(i)> \pi^{-1}(i+1)\}.\]
\[ \Des(T):= \{ i \, : \, 1\le  i \le n-1 \textmd{ and } i+1 \textmd{ appears in a row below } i \textmd{ in } T\}.\]

For any $\pi$ in a given Knuth class for tableau $T$, we have \cite[Lemma 2.10]{TaskinInner}
\[ \Des_L(\pi) = \Des(T).\] 
In other words, the left descent
set is constant on Knuth classes.

We can read the left descent set of a permutation directly off of its sequence of intervals.
\begin{example}
\label{E:PermSILeftDescent}
    For the permutation $\pi=(2537416)\in S_7$, we can see that the left descent set $\Des_L(\pi)=\{1,4,6\}$ of $\pi$ coincides 
    with the set of positions where an interval in $\mathrm{SI}(\pi)=(({\textbf{1}}),(2,3,{\textbf{4}}),(5,{\textbf{6}}),(7))$ ends and a new one begins (note the bolded numbers).
\end{example}
Using the one-to-one correspondence $\Phi_n$ between the set of all non-crossing set partitions of $[n]$ and subsets of $[n-1]$ given by:
\[ \Phi_n (\{s_1,s_2,\cdots,s_k\}) = \{\max s_1,\max s_2,\cdots, \max s_{k-1}\}, \]
we encode our observation in the following proposition:
\begin{proposition}\label{sidesPhi}
Let $\pi \in S_n$ be given. Then $\Des_L(\pi) = \Phi_n(\mathrm{SP}(\mathrm{SI}(\pi))$.
\end{proposition}
Proposition \ref{P:SIsameasSHS}
allows us to restate Proposition \ref{sidesPhi} in terms of tableaux:
\begin{proposition}
For any $S \in SYT_n$,  $\Des(S) = \Phi_n(\mathrm{SP}(\SHS(S)))$.
\end{proposition}

This implies that the chain-strip order can also be defined in terms of descent sets instead of SHS:

\begin{proposition}

Let $S, T \in SYT_n$. Then $S \le_{C-S} T$ if and only if we can find a sequence of tableaux $S=S_0$, $S_1$, $S_2$, $\cdots$, $S_k = T$ in $SYT_n$ such that for each $i < k$, $S_i \cho S_{i+1}$ and  $\Des(S_i) \subseteq \Des(S_{i+1})$ with $|\Des(S_{i+1})| - |\Des(S_i)| \le 1$.
\end{proposition}

\section{Properties of the Chain-Strip Order} \label{properties of the order}

We begin this section with a few definitions pertaining to sequences of horizontal strips inspired by some standard constructions on permutations and tableaux (\S\S\ref{SS:SHSConstructions}). Using these, we prove a handful of desirable properties of the chain-strip order (\S\S\ref{SS:C-SProperties}). We end the section with some partial results about how the chain-strip order relates to other well-known partial orders defined on $SYT_n$ (\S\S\ref{comparing}).

\subsection{Some constructions on SHS inspired by permutation operations}
\label{SS:SHSConstructions}
We begin with  definitions of some natural constructions on permutations:

\begin{definition}\label{D:RevEvacPerm}
For a permutation $\pi=x_1x_2...x_n$, we define a new permutation by reversing the position of numbers, that is, we define
$$\pi^{rp}=x_nx_{n-1}...x_1 $$ and we define a permutation by reversing the values of $\pi$, that is we define 
$$\pi^{rv}=(n+1-x_1)(n+1-x_{2})...(n+1-x_n).$$ 
Finally, we define the {\it evacuation}  $\pi^{evac}$ of $\pi$ as $\pi^{evac} =(\pi^{rv})^{rp}= (\pi^{rp})^{rv}=(n+1-x_n)(n+1-x_{n-1})...(n+1-x_1)$.\footnote{The evacuation of a permutation can also be defined alternatively as follows: for $\sigma \in S_n$, $\sigma^{evac} = \omega_0 \cdot \sigma \cdot \omega_0$, where $\omega_0$ denotes the longest element of $S_n$.}
\end{definition}

We make analogous definitions for sequences of horizontal strips:

\begin{definition} \label{evacshs}
 Let  
  $s=(s_1,s_2,\cdots,s_k)$
 be a sequence of horizontal strips of total length $n$ (that is, $|s_1|+|s_2|+\cdots+|s_k|=n$). 
 \begin{enumerate}
     \item The {\it reverse} of $s$ is the sequence of horizontal strips $s^r$  which is defined by the rule that whenever $i$ and $i+1$ lie in the same strip of $s$, they lie in two different strips of $s^r$, and vice versa. 
     \item The {\it evacuation} of $s$ is the sequence of horizontal strips $s^{evac}$ which is obtained by replacing every integer  $i$ with
 $n+1- i$, reordering the integers in each strip from left to right in increasing fashion, and finally reordering the subsequences within the sequence to end up with a SHS.
 \end{enumerate}
\end{definition}

These definitions are compatible with one another in a most natural way.

\begin{example}
Let $\pi= 7342516 \in S_7$. Then $s = \SHS(\pi) = ((1),(2),(3,4,5,6),(7))$. Using Definition \ref{D:RevEvacPerm}, we can compute: 
  $$\pi^{rp}= 6152437 \quad \text{ and }  \quad \pi^{evac}= 2736451.$$
   And using Definition \ref{evacshs}, we have 
   \[ s^r=((1,2,3),(4),(5),(6,7)) \quad  \textmd{ and } \quad s^{evac}=((1),(2,3,4,5),(6),(7)).\]
\end{example}

In the example above, note that $s^r = \SHS(\pi^{rp})$ and $s^{evac}= \SHS(\pi^{evac})$. More generally we have:

\begin{lemma} \label{shsinvpir}
Let $\pi \in S_n$ be given. Then    $\SHS(\pi^{rp})=\SHS(\pi)^r$ and  $\SHS(\pi^{evac})=\SHS(\pi)^{evac}.$
\end{lemma}
\begin{proof}
For any $1 \leq i < n$, $i, i+1$ are in the same horizontal strip of $\pi$ iff $i$ occurs before $i+1$ in $\pi$ iff $i+1$ occurs before $i$ in $\pi^{rp}$. This corresponds to $i, i+1$ being in the same horizontal strip of $\SHS(\pi)$ iff they are not in the same horizontal strip in $\SHS(\pi)^{r}$, proving $\SHS(\pi^{rp})=\SHS(\pi)^r$. For the second assertion, it suffices to note that $i+1$ comes after $i$ in $\pi$ iff $n+1-(i+1)$ comes before $n+1-i$ in $\pi^{evac}$. Then expanding definitions as we did earlier completes the proof. 
\end{proof}

Definitions \ref{D:RevEvacPerm}-\ref{evacshs} are motivated by analogous constructions on tableaux. 
  Recall that for a standard Young tableau $T \in SYT_n$, the {\it transpose} $T^t$ of $T$ is the standard Young tableau obtained by reflecting the entries in $T$ along the main diagonal. The definition of the {\it evacuation} $T^{evac}$ of $T \in SYT_n$ is somewhat more involved; see \cite{Sagan} for details. 
 For us it will suffice to note that for $S \in SYT_n$, if $\sigma \in \kappa_S$, then $\sigma^{evac} \in \kappa_{S^{evac}}$. More generally we will need the following theorem, proved by Schensted  \cite{Schensted}:
\begin{theorem}[Schensted] \label{schensted}
Let $\pi \in S_n$, and let  $P(\pi)$ be the insertion tableau of $\pi$. Then $P(\pi^{evac})=P(\pi)^{evac}$ and $P(\pi^{rp})=P(\pi)^t$. 
\end{theorem}

In the next section we {use} this theorem, together with Lemma \ref{shsinvpir}, as we talk interchangeably about operations like evacuation and reversion / transposition acting on permutations $\pi$, tableaux $P = P(\pi)$, and associated sequences of horizontal strips $s = \SHS(P) =\SHS(\pi)$.

\subsection{Some basic properties of the chain-strip order and their implications}
\label{SS:C-SProperties}

A partial order $\le$ on $SYT_n$ is said to {\it restrict to segments} if 
 $$S \leq T~~\text{ implies }~~S_{[i,j]} \leq
 T_{[i,j]}~~ \text{ for all } ~~1\leq i<j \leq n.
 $$
Conversely such a partial order $\le$ is said to {\it extend from segments} if 
 $$S \leq T~~\text{ implies }~~\Omega_1(S) \leq
 \Omega_1(T) ~~ \text{ and } ~~\Omega_2(S) \leq
 \Omega_2(T),
 $$
where the maps $\Omega_1$, $\Omega_2$ are defined as follows: For each $T
\in SYT_n$, $\Omega_1:SYT_n \mapsto SYT_{n+1}$ concatenates $n+1$ to
the first row of $T$  from  the right whereas $\Omega_2:SYT_n
\mapsto SYT_{n+1}$ concatenates $n+1$ to the first column  of $T$
from  the bottom.

\newpage
Then the next result follows from definitions and the second result can be easily proven:
\begin{proposition}
\label{P:ResSegC-S}
The chain-strip order restricts to segments. That is, for $S,T \in SYT_n$, we have
 $$S \leq_{C-S} T~~\text{ implies }~~S_{[i,j]} \leq_{C-S}
 T_{[i,j]}~~ \text{ for all } ~~1\leq i<j \leq n.
 $$
\end{proposition}
\begin{proposition}
\label{P:ExtSegC-S}
The chain-strip order extends from segments. That is, $$S \leq_{C-S} T~~\text{ implies }~~\Omega_1(S) \leq_{C-S}
 \Omega_1(T) ~~ \text{ and } ~~\Omega_2(S) \leq_{C-S}
 \Omega_2(T),
 $$
\end{proposition}

We put together two more results about maps preserving the chain-strip order in the next proposition:

\begin{proposition}
The following maps are order-preserving:
\begin{enumerate}
\item[(i)]  The map
$$
(SYT_{n},\leq_{C-S}) \rightarrow  (2^{[n-1]}, \subseteq)
$$
sending a tableau $S$ to its descent set $\Des(S)$.
\item[(ii)] The map
$$
(SYT_{n},\leq_{C-S}) \rightarrow (\Par_{n},\leq^{op}_{dom})
$$
sending $S$ to its shape $\sh(S)$.
\end{enumerate}
\end{proposition}

\begin{proof}
 In Section \ref{descreint} we described the relationship between SHS and descent sets. In essence,  \textit{i+1} occurs below \textit{i} in $S$ exactly when a new horizontal strip begins in its growth chain. Thus if $S,T \in SYT_n$ are related in the chain-strip order, say, $S \le_{C-S} T$, they must subsequently have containment of descent sets, that is, $\Des(S) \subseteq \Des(T)$, in turn implying that the map sending a tableau to its descent set is order-preserving.\footnote{ This is also implied by Lemma 3.2 of \cite{Taskin} because the chain-strip order is stronger than the weak order (Proposition \ref{weakimpliesSHS}) and restricts to segments (Proposition \ref{P:ResSegC-S}).} 
The second one follows directly from the shape component of Definition \ref{NewSHSOrderDef}.
\end{proof}

Next we consider two other natural operations defined on $SYT_n$ and see how they interact with the chain-strip order. We begin with transposition:

\begin{proposition}
The map
$$
(SYT_{n},\leq_{C-S}) \rightarrow  (SYT_{n},\leq_{C-S})
$$
sending $T$ to its transpose $T^t$ is a poset anti-automorphism.
\end{proposition}

\begin{proof}
We know that the map sending $T$ to its transpose is a poset anti-automorphism of the chain order \cite{Taskin}, so we only need to check the horizontal strip part of Definition \ref{NewSHSOrderDef}.  It suffices to prove the statement in the case of some $S$ and $T$ differing by at most a one-step refinement. Let $\pi \in \kappa_S$ for $S \in SYT_n$. Then by Theorem \ref{schensted}, $\pi^{rp} \in \kappa_{S^t}$, and by Lemma \ref{shsinvpir}, $\SHS(S^t)=(\SHS(S))^r$. So for $T \in SYT_n$ such that $S{\leq}_{C-S}T$, let us assume that $\SHS(T)$ is at most a one-step refinement of $\SHS(S)$. Then by the definition of the reverse of a SHS (Definition \ref{evacshs}), $\SHS(T^t)$ would split consecutive strips if they are joined in $\SHS(T)$ and join them if they are split in $\SHS(T)$, and likewise for $\SHS(S)$. Subsequently $\SHS(S^t)$ has to be at most a one-step refinement of $\SHS(T^t)$, yielding $T^t\leq_{C-S}S^t.$ 
\end{proof}

Now we consider the evacuation map \cite{Sagan, Schutzenberger1}. {W}e begin with an {illustrative} example. 

\begin{example}
    Let $S, T \in SYT_6$ be: \[  S= \tableau{1&2&5\\3&4&6} \qquad \textmd{ and } \qquad T= \tableau{1&4&5\\2&6\\3}. \]  Then $\SHS(S)=((1,2),(3,4,5),(6)) $  and $ \SHS(T)=((1),(2),(3,4,5),(6)),$
    implying that $\SHS(T)$ is a one-step refinement of $\SHS(S)$ 
   and  minimal calculations on shapes of the restricted parts yield $S {\leq}_{C-S} T$. Then we compute: 
\[  S^{evac}= \tableau{1&3&4\\2&5&6}  \qquad \textmd{ and } \qquad T^{evac}= \tableau{1&3&4\\2&5\\6}, \]
together with   $\SHS(S^{evac})=((1),(2,3,4),(5,6))$  and $\SHS(T^{evac})=((1),(2,3,4),(5),(6)).$ 
In other words, $\SHS(T^{evac})$ is a one-step refinement of $\SHS(S^{evac})$, which implies that $S^{evac} \le_{C-S} T^{evac}$.
     \end{example}
    
This motivates the following proposition:
\begin{proposition}

 The Sch\"utzenberger's evacuation map
$$
(SYT_{n},{\leq}_{C-S}) \rightarrow  (SYT_{n},{\leq}_{C-S})
$$
sending $T$ to its evacuation tableau $T^{evac}$ is a poset
automorphism.
\end{proposition}

   \begin{proof} We know that the map sending $T$ to its evacuation tableau is a poset automorphism of the chain order \cite{Taskin}, so we only need to check the horizontal strip part of Definition \ref{NewSHSOrderDef}.
   
   Suppose that $S \le_{C-S} T$. Let $\sigma \in \kappa_S$ and $\tau \in \kappa_T$.  If $\SHS(\sigma)=\SHS(\tau)$, then by Lemma \ref{shsinvpir}, we have $\SHS(\sigma^{evac})=\SHS(\sigma)^{evac}=\SHS(\tau)^{evac}= \SHS(\tau^{evac})$. So we only need to consider the case where $\SHS(\tau)$ is a one-step refinement of $\SHS(\sigma)$. The transitive closure case will follow automatically.

   We suppose, therefore, that $\SHS(\tau)$ is a one-step refinement of $\SHS(\sigma)$, implying that the two SHS are identical except for one break at some location $i$, so that $i$ and $i+1$ are in the same strip of $\SHS(\sigma)$ while they are in two distinct strips of $\SHS(\tau)$. 
   
    Recall from Definition \ref{evacshs} that 
   evacuation of an SHS
   corresponds to replacing every integer  $i$ with
 $n+1- i$, reordering the integers in each strip from left to right in increasing fashion, and finally reordering the subsequences within the sequence to end up with a new SHS.
    This implies that $\SHS(\sigma^{evac})$
    and $\SHS(\tau^{evac})$ are identical except for one break at location 
    $n-i$. 
  That is, $\SHS(\tau^{evac})$ is a one-step refinement of $\SHS(\sigma^{evac}$). Then Theorem \ref{schensted} allows us to conclude that $\SHS(T^{evac})$ is a one-step refinement of $\SHS(S^{evac})$ and we are done.
    \end{proof}

We wrap up this section with a neat implication of the properties of the chain-strip order which we already proved. 
Recall that the weak order can be used to describe the multiplicative structure of the Poirier–Reutenauer Hopf algebra that was originally described in terms of jeu de taquin slides \cite{Poirier-Reutenauer}; moreover, the same structure can be described by any other partial order on $SYT_n$ satisfying a handful of properties \cite{Taskin}. 
It turns out that the chain-strip order is just one such order:

\begin{proposition}
The multiplicative structure of the  Poirier-Reutenauer Hopf algebra can be defined using $\leq_{C-S}$ as follows: for $S \in SYT_k$ and $T \in SYT_l$ with $k+l = n$
 $$ S \ast T = \sum_{\substack{R \in SYT_n:\\
S\slash T \leq_{C-S} R \leq_{C-S} S\backslash T}} R, $$
where $S \slash T$ (respectively $S\backslash T$)  denotes the tableau in $SYT_n$ whose columns (respectively rows) are obtained by concatenating the columns (respectively rows) of ${T}^{+k}$, the partial tableau obtained as a result of adding $k$ to every entry of $T$, to the bottom of the columns of $S$. This operation is commutative because of the construction of the Hopf algebra.

\end{proposition}

\begin{proof}
As the chain-strip order is stronger than the weak order (Propositions \ref{weakimpliesSHS}) and it restricts to segments (Proposition \ref{P:ResSegC-S}),
this follows from Theorem 4.2 of \cite{Taskin}.
\end{proof}

As the chain-strip order is particularly straightforward to work with, it has potential as a computational aid to those working with this Hopf algebra.

\subsection{Comparing with other known orders} \label{comparing}

It is known \cite{Melnikov4, Taskin} that the four partial orders on $SYT_n$, the weak, Kazhdan-Lusztig, geometric, and chain orders, satisfy the following:
$$\text{weak order}\subsetneq \text{Kazhdan-Lusztig (KL) order} \subset \text{geometric order} \subsetneq \text{chain order.}$$
In other words, the weak order is weaker than the KL, geometric, and chain orders, and alternatively, the chain order is  stronger than all three.
It is also known that these four orders on $SYT_n$ coincide for $n \leq 5$ and for $n \geq 6$, they differ.

We have already seen that the chain-strip order is stronger than the weak order (Proposition \ref{weakimpliesSHS}). It is a natural question to ask how close we get. 

It turns out that for $n\le 6$ the chain-strip order coincides with the weak order. For $n=7$, it 
is the same as the weak order except for four pairs (which are connected through transposition and evacuation) of newly comparable tableaux for our order. That is, for $S$ and $T$ given as the following two tableaux from $SYT_7$:
$$S=\tableau{1&3&4&5\\2&7\\6} \qquad \text{ and } \qquad  T=\tableau{1&3&4\\2&5&7\\6},$$
we have $S \le_{C-S} T$ but $S \not <_{weak} T$. Similarly we have that $S^{evac} \le_{C-S} T^{evac}$, $T^t \le_{C-S} S^t$, and $T^{{evac}^t} \le_{C-S} S^{{evac}^t}$, while these same pairs are not comparable in the weak order.

Interestingly, if we modify our order to include at most two-step refinements of the associated sequences of horizontal strips, then for $n=7$, these same four pairs are the only difference between this new order and the KL and geometric orders---and so equivalently, the only difference between this new order and Melnikov's two partial orders (Duflo-Vogan and Vogan-Chain cf.\ \cite{Melnikov4}), as the latter are all equal to KL and geometric orders until $n=10$. 

The above discussion shows that constraints on SHS (or, given Section \ref{descreint}, descent conditions) are a flexible and easily adjustable tool that might be worth pursuing in order to get a better combinatorial handle on the many partial orders defined on $SYT_n$. Furthermore, for $n\le 6$, the agreement of the chain-strip order with the weak order can significantly streamline computations. Note also that while the complexity of other partial orders on $SYT_n$ grows with each $n \geq 7$, the chain-strip order and subsequent SHS / descent computations remain relatively manageable.
Thus if we can understand better the discrepancies between the chain-strip order and other partial orders, we can use the simplicity of the chain-strip order to carry out computations and analyses of other partial orders in an easier context.

\section{Remaining questions and further work} \label{remaining questions}

Partial orders that are stronger than the weak order satisfy a handful of interesting properties. For example, if such a partial order satisfies certain additional conditions, the subposet $SYT_m^R$~ ~of~
~$SYT_m$ defined by:
$$SYT_m^R=\{T\ \in SYT_m \mid T_{[1,k]}=R\},$$
for some $R \in SYT_k$ with $m=k+n$,  has a really simple characterization \cite{Taskin}. The chain-strip order satisfies these conditions. Recalling the definitions of $\hat{0}_{R,n}$ and $\hat{1}_{R,n}$ as the following:

 $\put(0,100){\makebox(1,1){}}
 \put(42,50){\makebox(12,35){$\hat{0}_{R,n}$=}}
 \put(85,65){$R$}
\put(65,90){\line(1,0){70}}
 \put(105,60){\line(1,0){30}}
\put(65,40){\line(1,0){40}} \put(135,60){\line(0,1){30}}
\put(105,40){\line(0,1){20}} \put(65,40){\line(0,1){50}}
\put(135.5,74.5){\framebox(57,15){{\scriptsize $k\!\!+\!\!1~\ldots~k\!\!+\!\!n$}}}
\put(232,50){\makebox(12,25){$\hat{1}_{R,n}$=}} \put(275,65){$R$}
\put(255,90){\line(1,0){70}}
 \put(295,60){\line(1,0){30}}
\put(255,40){\line(1,0){40}}
 \put(325,60){\line(0,1){30}}
\put(295,40){\line(0,1){20}}
 \put(255,40){\line(0,1){50}}
\put(255.5,4){\framebox(19,36){\shortstack{{\scriptsize
$k\!\!+\!\!1$}\\\\.\\.\\.\\{\scriptsize $k\!\!+\!\!n$}}}}$

\noindent
we can state the following explicitly:
\begin{proposition} Let $R \in SYT_k$ with $m=k+n$ be given.
Then, with respect to the chain-strip order $\leq_{C-S}$, the 
subposet $SYT_m^R$~ ~of~
~$SYT_m$ is precisely the interval $[\hat{0}_{R,n},\hat{1}_{R,n}]$.
Furthermore, 
 the  proper part of $SYT_m^R$  is homotopy equivalent to
  $$
\begin{cases}
\text{ an } (n-2)-\text{dimensional sphere} & \text{ if } R \text{ is
rectangular,} \\
\text{ a point }                          & \text{ otherwise.} \\
\end{cases}
$$\end{proposition}

\begin{proof}
As the chain-strip order is stronger than the weak order (Propositions \ref{weakimpliesSHS}), restricts to segments (Proposition \ref{P:ResSegC-S}), and extends from segments (Proposition \ref{P:ExtSegC-S}), this follows from Proposition 7.1 of \cite{Taskin}.
\end{proof}

However, there are other properties of the chain-strip order that we have yet to explore. In this section we look at two specific properties.

\subsection{Robinson-Schensted Insertion Property}

Given a Young tableau $T \in SYT_n$ and $a\in \{1,\ldots,n+1\}$, let $\ov
T_a$ be the partial tableau obtained from $T$ by adding $1$ to all entries greater than or equal to $a$.
A partial order $\leq$ on $SYT_n$ is said to have the \textit{Robinson-Schensted (RS) insertion property} (first introduced by \cite{Melnikov4}) if  
 $S,T\in SYT_n$ are such that $S \leq T$, 
then for any $a\in\{1,\ldots,n+1\}$ one has $c_a( \ov S_a)\leq 
c_a( \ov T_a)$ and $r_a(\ov S_a) \leq 
r_a(\ov T_a).$

While the weak and Kazhdan-Lusztig orders have the RS insertion property, the chain order does not \cite{Melnikov4}. We can use an example from \cite{Melnikov4} to show that this property is also not satisfied by the chain-strip order. Let $S, T \in SYT_7$ be given as follows:
$$
T= \tableau{1&2&6\\3&5\\4&7} \qquad {\rm and}\qquad
S= \tableau{1&2&6\\3&7\\4\\5}.
$$

We have $T \leq_{C-S} S$ because  $T\chos S$ and  $\SHS(S)= ((1,2),(3),(4),(5,6),(7))$ is a one-step refinement of 
$((1,2),(3),(4,5,6),(7)) =\SHS(T)$. However, for
$$
c_5(\ov T_5)= \tableau{1&2&7\\3&6\\4&8\\5} \qquad {\rm and}\qquad c_5(\ov S_5)=
 \tableau{1&2&7\\3&6&8\\4\\5},
$$
$\SHS(c_5(\ov T_5))= ((1,2),(3),(4),(5,6,7),(8))=\SHS(c_5(\ov S_5))$, 
but $c_5(\ov T_5)\not\cho c_5(\ov S_5)$, and so $c_5(\ov T) \not \leq_{C-S} c_5(\ov S)$.

We note here that neither row nor column RS insertion changes the SHS of a tableau in a way that would lead to a violation of the SHS condition in Definition \ref{NewSHSOrderDef}. That is, given $S,T$ with $S \leq_{C-S} T$, and $\SHS(T)$ being at most a one-step refinement of $\SHS(S)$, it is easily seen by the rules of RSK insertion that $r_a(\ov T_a)$ (respectively, $c_a(\ov T_a)$) is at most a one-step refinement of $r_a(\ov S_a)$ (respectively, $c_a(\ov S_a)$).

The question of how to extend the definition of the chain-strip order so as to incorporate the RS insertion property remains a task for further work.

\subsection{Inner tableau translation property on our order}

A partial order $\leq$ on $SYT_n$ is said to have the {\it inner translation property} if  for any $S,T \in SYT_n$ with $S \le T$, $S^{\prime}\leq T^{\prime}$,
where the tableaux $S^{\prime}$ and $T^{\prime}$ are obtained by applying to $S$ and $T$ a single dual Knuth relation on some triple $(i-1,i,i+1)$ of consecutive numbers with $1 < i < n$.\footnote{ \label{FN:DualKnuth} We say  $\sigma, \tau \in S_{n}$ differ by a single  {\it dual Knuth relation},  if for some consecutive integers $1\leq i<i+1<i+2\leq n$,
$\sigma= \ldots i \ldots i+2 \ldots i+1 \ldots \textmd{ and }
	\tau=\ldots i+1 \ldots i+2 \ldots i \ldots\ldots x_{n}$
     or
$     \sigma= \ldots i+1 \ldots i \ldots i+2 \ldots \textmd{ and }
	\tau=\ldots i+2 \ldots i \ldots i+1 \ldots\ldots x_{n}$.
     Two permutations are called  {\it dual Knuth
		equivalent}, written $\sigma  ~\dKnuth~ \tau$, if one  can be obtained
	from the other by applying a sequence of  dual Knuth relations. It is not hard to see that $\sigma  ~\dKnuth~ \tau$ if and only if $\sigma^{-1}  ~\Knuth~ \tau^{-1}$. See 
	\cite{Sagan} for more details.} While the KL and geometric orders have the inner translation property, the weak and chain do not \cite{Melnikov4, Taskin}. It is then natural to expect that the chain-strip order will not, either. We illustrate with an example from \cite{Taskin} how the SHS / descent constraints can contribute to such failure:

\begin{example}
    Let $S,T \in SYT_6$ with  $\SHS(S)=((1,2),(3,4),(5,6))=\SHS(T)$ be given as follows:
$$
S= \tableau{1&2&4\\3&5&6} \qquad {\rm and}\qquad
T= \tableau{1&2&4\\3&6\\5}.
$$
Minimal calculations on the restrictions of the shapes of $S$ and $T$ yield the result $S \leq_{C-S} T$. However, after applying a single dual Knuth relation on the triple $\{3,4,5\}$, we get:
$$
S^{\prime}= \tableau{1&2&3\\4&5&6} \qquad {\rm and}\qquad
T^{\prime}= \tableau{1&2&5\\3&6\\4}.
$$
As $\SHS(T^{\prime})=((1,2),(3),(4,5),(6))$ is a two-step refinement of $\SHS(S^{\prime})=((1,2,3),(4,5,6))$, and there is no intermediate shape between $\sh(S^{\prime})$ and $\sh(T^{\prime})$, $S^{\prime} \nleq_{C-S} T^{\prime}$. 

\end{example}

    Although the chain-strip order does not satisfy the inner translation property, one might wonder if it might satisfy a slightly restricted version. 
The \emph{inner tableau translation property} was first defined in \cite{Taskin}:

\begin{definition}
Let $R, \tilde{R} \in SYT_k$ be two tableaux with the same shape, $k < n$. For $S,T \in SYT_n$ having the same inner tableau $R$ (that is, $S_{[1,k]} = T_{[1,k]} = R$), let $\tilde{S}$ (respectively, $\tilde{T}$) be the tableau obtained by replacing $R$ with $\tilde{R}$ in $S$ (respectively in $T$). Then we say that a given partial order $\leq$ on $SYT_n$ has the \emph{inner tableau translation property} if $$S \leq T \text{ implies } \tilde{S} \leq \tilde{T}.$$
\end{definition}

It is not known whether the weak order has the inner tableau translation property, but if it does, then we know that this can be used to provide a combinatorial proof of the fact that  the weak order is well-defined  \cite{TaskinInner}. Moreover  this  property plays a crucial role in  understanding  the structure of the Poirier-Reutenauer Hopf algebra \cite{TaskinInner}.

    However it turns out that the chain order, and consequently the chain-strip order, cannot have the inner {tableau} translation property. Let $S,T \in SYT_7$ be given as follows:  $$S=\tableau{1&2&6\\3&5\\4&7} \qquad \text{ and } \qquad T=\tableau{1&2&6\\3&7\\4\\5}.$$
    We have $S\cho T$  \cite[Section 3.9]{Melnikov4}, and as $\SHS(T)= ((1,2),(3),(4),(5,6),(7))$ is a one-step refinement of $\SHS(S) = ((1,2),(3),(4,5,6),(7))$, we also have $S \le_{C-S} T$.

    Now if we define $R \in SYT_4$ as $R = S_{[1,4]} = T_{[1,4]}$ and set $\tilde{R} \in SYT_4$ to be the result of the single dual Knuth move on $(1,2,3)$ in $R$, then we have:
$$\tilde{S}=\tableau{1&3&6\\2&5\\4&7} \qquad \text{ and } \qquad \tilde{T}=\tableau{1&3&6\\2&7\\4\\5}.$$
We can now see that $\tilde{S}\stackrel {\rm  Ch}{\nleq} \tilde{T}$ (and consequently, $\tilde{S} \nleq_{C-S} \tilde{T}$) 
by noting that the opposite dominance condition on the shapes of all restrictions (as in Proposition \ref{P:ChainOneLiner}) does not hold:\ in particular $\sh(\tilde{S}_{[2,7]}) = \sh(\tilde{T}_{[2,7]})$\ but $\tilde{S}_{[2,7]} \neq  \tilde{T}_{[2,7]}$.
    
This example captures a fundamental failure of the chain order in handling hook-shaped tableaux during the tableau translation process---the same problem does not come up with a four-box tableau of shape $(2,2)$.
The fact that the weak order is preserved by tableau translations involving hooks \cite[Proposition 4.7]{TaskinInner} points towards an intriguing divergence between the chain order and the weak order.

\end{document}